\documentclass[reqno,11pt]{amsart}

\usepackage{mathdots}
\usepackage{tikz}
\usepackage{tikz-cd}
\usepackage{amssymb}
\usepackage{amsgen}
\usepackage{amsmath}
\usepackage{amsthm}
\usepackage{mathrsfs}
\usepackage{cite}
\usepackage{amsfonts}

\newcommand{\LM}{\mathsf{LM}}
\newcommand{\GGM}{\mathsf{GGM}}

\newcommand{\RM}{\mathsf{RM}}

\newcommand{\J}{\mathrel{\mathscr J}} 
\newcommand{\R}{\mathrel{\mathscr R}} 
\newcommand{\eL}{\mathrel{\mathscr L}} 
\newcommand{\HH}{\mathrel{\mathscr H}}

\newcommand{\inv}{^{-1}}
\newcommand{\p}{\varphi}

\usepackage{xcolor}

\newtheorem{Thm}{Theorem}[section]
\newtheorem{Prop}[Thm]{Proposition}

\newtheorem{Lemma}[Thm]{Lemma}
{\theoremstyle{definition}
}
{\theoremstyle{remark}
\newtheorem{Rmk}[Thm]{Remark}}
\newtheorem{Cor}[Thm]{Corollary}
{\theoremstyle{remark}
}
{\theoremstyle{remark}
}

{\theoremstyle{remark}
}
{\theoremstyle{remark}
}
{\theoremstyle{remark}
}
{\theoremstyle{remark}
\newtheorem*{Claim*}{Claim}}

\numberwithin{equation}{section}

\title[The minimal faithful degree of Rhodes semisimple semigroups]{On the minimal faithful degree of Rhodes semisimple semigroups}

\author{Stuart Margolis\and Benjamin Steinberg}
\date{\today}
\thanks{The second author was supported by a Simons Foundation Collaboration Grant, award number 849561.}

\begin{document}

\begin{abstract}
In this paper we compute the minimal degree of a faithful representation by partial transformations of a finite semigroup admitting a faithful completely reducible matrix representation over the field of complex numbers.  This includes all inverse semigroups, and hence our results generalize earlier results of Easdown and Schein on the minimal faithful degree of an inverse semigroup.  It also includes well-studied monoids like full matrix monoids over finite fields and the monoid of binary relations (i.e., matrices over the Boolean semiring).   Our answer reduces the computation to considerations of permutation representations of maximal subgroups that are faithful when restricted to distinguished normal subgroups.  This is analogous to (and inspired by) recent results of the second author on the minimal number of irreducible constituents in a faithful completely reducible complex matrix representation of a finite semigroup that admits one.  To illustrate what happens when a finite semigroup does not admit a faithful completely reducible representation, we compute the minimal faithful degree of the opposite monoid of the full transformation monoid.
\end{abstract}

\maketitle

\section{Introduction}

In this paper semigroups act on the right of a set. All semigroups in the paper are finite semigroups. Cayley's Theorem embedding a group $G$ into the symmetric group $S_{|G|}$ is one of the oldest theorems in group theory. The proof generalizes mutatis mutandis to
prove that if $M$ is a monoid, then $M$ embeds into the monoid $T_{|M|}$ of all functions from the underlying set of $M$ to itself (acting on the right). For
semigroups, one has to be a bit careful. First embed $S$ into the monoid $S^{1}$, where $1$ is an externally added identity element, and then $S$ is isomorphic to a subsemigroup of the image of $S^{1}$ in $T_{(|S^{1}|)}$ in the embedding above. The example of a left-zero semigroup of order $2$ shows that the addition of the identity element is sometimes necessary.

Recall that an inverse semigroup is a semigroup $S$ such that for all $s \in S$, there is a unique $s^{-1} \in S$, called the inverse of $s$, such that $ss^{-1}s=s$ and $s^{-1}ss^{-1}=s^{-1}$. Equivalently, it is a semigroup such that for all $s \in S$, there is a $t \in S$ such that $sts=s$ and such that the set $E(S)$ of idempotents of $S$ forms a semilattice, that is, every pair of idempotents of $S$ commute. The symmetric inverse monoid $SIM_n$, consisting of all partial bijections on a set of size $n$ to itself, plays the role of the symmetric group in inverse semigroup theory. It is indeed an inverse semigroup in which the inverse $f^{-1}$ of $f \in SIM_n$ is the usual set theoretic inverse partial function. The Preston-Wagner Theorem embeds any inverse semigroup $S$ into $SIM_{|S|}$. The proof is similar to that of Cayley's Theorem, but one must be clever in restricting the domains of elements in the Cayley embedding in order to get an embedding into the symmetric inverse monoid.

This leads naturally to the question of minimal embeddings of finite semigroups into full transformation monoids. If $S$ is a semigroup, we define $\mu(S)$ to be the least $n$ such that $S$ is isomorphic to a subsemigroup of $T_n$.

We will also be interested in minimal representations of semigroups by partial functions, that is, by minimal size embeddings into the monoid $PT_n$ of all partial functions on an $n$-element set (again acting on the right). We write $m(S)$ for this value. Clearly $T_n$ is a subsemigroup of $PT_n$. Furthermore
$PT_n$ can be identified with the stabilizer of $n+1$ in $T_{n+1}$ by completing $f \in PT_n$ to an element of $T_{n+1}$ by sending $n+1$ and all elements not in the domain of $f$ to $n+1$. Therefore, $m(S) \leq \mu(S) \leq m(S)+1$. If $S$ is an inverse semigroup, it is known that $m(S)$ is the least integer $n$ such that $S$ embeds into the symmetric inverse monoid $SIM_n$~\cite{petrich}.  If $G$ is a group, then it is not difficult to see that $\mu(G)=m(G)$ is the least $n$ such that $G$ embeds into the symmetric group $S_n$.  We use this fact without further mention.

The problem of determining $m(S)$ has proved to be quite difficult. The first paper on the subject for groups gave some foundational results and solved the problem for finite abelian groups~\cite{minimalperm}. If $G$ is a simple group, it is easy to see that the minimal degree of a faithful representation comes from a transitive representation and thus the problem of computing $\mu(G)$ is the same as finding a proper subgroup of minimal index. See~\cite{mindegsimple} and its bibliography for results for simple groups. More generally if $G$ is a subdirectly irreducible group then the minimal degree of a faithful representation is afforded by a transitive representation. In particular, this is true for all symmetric groups $S_n$, a result we use in this paper. For more results for groups see~\cite{easdownsemidirect} and its bibliography.

The first paper concerning minimal degrees of semigroups is Easdown's~\cite{Easdownfaithful}, where the minimal degree of a semilattice of groups is described. Schein~\cite{scheininvmindeg} reduces the problem of computing minimal degrees of inverse semigroups to a related problem in group theory. Given the difficulty of the problem for groups, this is the most general result that one can expect for inverse semigroups. The only paper beyond the class of inverse semigroups that we are aware of is~\cite{mindegrectband}.

In this paper we significantly extend the work beyond inverse semigroups. A key feature of representations of inverse semigroups and in particular groups is that every representation by partial bijections is a disjoint union of transitive representations. This is far from true for general semigroups. In this paper, we introduce the class of Rhodes semisimple semigroups and prove that the minimal degree of a faithful representation by (partial) functions is given by a disjoint union of transitive representations. This class includes, besides inverse semigroups, the monoid of all partial functions on a set and the monoid of all $n \times n$ matrices over a field and many more.

We do this by relativizing results from the representation theory of finite monoids over the complex numbers to the category of representations by functions and partial functions. In~\cite{Rhodeschar} Rhodes determined the smallest congruence $\equiv$, now called the (complex) Rhodes radical~\cite{repbook}, on a finite semigroup such that $S/{\equiv}$ has a faithful semisimple (i.e., completely reducible) representation over the complex numbers. In particular, a finite semigroup $S$ has a faithful semisimple complex representation if and only if $\equiv$ is the identity relation. We call such a semigroup a Rhodes semisimple semigroup. We prove that for this class indeed every minimal degree faithful representation by (partial) functions is given by a disjoint union of transitive representations. The results in this section depend on the semilocal theory of finite semigroups. We give all relevant details in the appendix to this paper.

Furthermore, we use tools from non-additive homological algebra, such as tensor products of (partial) $S$-sets, $S$ a semigroup, to describe all transitive representations. This allows us to formulate a reduction of the problem to a group theoretic one. Inverse semigroups are Rhodes semisimple and we prove that our condition is equivalent to Schein's~\cite{scheininvmindeg} for inverse semigroups.

We also consider representations in the monoids $T_{n}^{op}$ and $PT_n^{op}$, the opposites of the monoids of (partial) functions on an $n$-set. Equivalently, these are representations of finite semigroups acting on the left of a set by (partial) functions. For ease of discussion we define $\lambda(S)=\mu(S^{op})$ and $l(S)=m(S^{op})$ where $S^{op}$ is the opposite semigroup of $S$. For inverse semigroups (and thus for groups), $\lambda(S)=\mu(S)$ and $l(S)=m(S)$ since the inverse operation is an anti-isomorphism. For semigroups in general, the best one can say is that $\lambda(S) \leq 2^{\mu(S)}$, since if $S$ acts on the left of a set $X$, it acts by inverse image on the right of
the power set of $X$. For $T_{n}^{op}$, we prove that this is the best one can do. That is $\lambda(T_{n}) = 2^{n}$. That is related to the known result that there are regular languages $L$ whose  minimal automaton has $n$ states but its reverse language $L^{rev}$ has a minimal automaton with $2^{n}$ states.

The paper is structured as follows. In Section~\ref{s:rss} we introduce the class of Rhodes semisimple semigroups and prove that disjoint unions of transitive representations suffice for this class to find the minimal degree of a faithful representation. We use the well-known fact that all transitive representations of a semigroup are images of Sch\"utzenberger representations on regular $\mathscr{R}$-classes and we use the second author's transformation analogue~\cite{transformations} of Green's condensation theory for linear representations~\cite{Greenpoly} to find maximal quotients of transitive representations that do not change the underlying permutation group associated to the maximal subgroup of the $\mathscr{R}$-class. We compute $\mu(M_{n}(F))$, the size of the minimal degree of a faithful representation of the monoid of all $n \times n$ matrices over a field $F$. For simple Rhodes semisimple semigroups with $2$ $\mathscr{R}$-classes and $2$ $\mathscr{L}$-classes we have an explicit formula that heavily depends on the theory of permutation groups. We completely compute these minimal degrees of faithful representations in case the maximal subgroup is a symmetric group and show how it depends on (the necessarily cyclic) maximal subgroup of the idempotent generated semigroup. Surprisingly, the outer automorphism of $S_6$ comes into play in giving the complete solution.

In the last section we look at the aforementioned problem of computing the minimal degree of a faithful representation of $S^{op}$ as a function of that of $S$. As mentioned above, we prove that the size of the minimal representation of $T_{n}^{op}$ is $2^{n}$. We close with the Appendix on semi-local theory of finite semigroups. All background in semigroup theory can be found in~\cite{qtheor}, especially Chapter~4 and the appendices. For background on the representation theory of finite monoids see~\cite{repbook}.

\section{Rhodes semisimple semigroups} \label{s:rss}

In this section we introduce the class of Rhodes semisimple semigroups. This uses tools from the semilocal theory of finite semigroups~\cite[Section 4.6]{qtheor},~\cite[Chapters 7-8]{Arbib}. For the reader's convenience we have listed the main results in appendix~\ref{s:appendix}.

Let $S$ be a semigroup.  Associated to any regular $\mathscr J$-class $J$ of $S$ are two congruences $\equiv_J$ and $\equiv_{\RM,J}$ defined as follows.  Put $s\equiv_J t$ if, for all $x,y\in J$, we have $xsy\in J\iff xty\in J$ and if they are both in $J$, then $xst=xty$. In the literature this is often denoted $\equiv_{\GGM,J}$ but we abbreviate it to $\equiv_J$ for ease of notation. $\GGM$ stands for generalized group mapping. See the Appendix for further explanation.
 Let $s\equiv_{\RM,J} t$ if, for all $x \in J$, $xs\in J\iff xt\in J$ and if they are both in $J$, then $xs=xt$.  Equivalently, $s\equiv_{\RM,J} t$ if they act the same in the right Sch\"utzenberger representation of $S$ on some (equivalently any) $\mathscr R$-class of $J$.  Clearly ${\equiv_{\RM,J}}\subseteq {\equiv_J}$.  Put ${\equiv_{\GGM}}=\bigcap {\equiv_J}$ and ${\equiv_{\RM}}=\bigcap {\equiv_{\RM,J}}$ and note that ${\equiv_{\RM}}\subseteq {\equiv_{\GGM}}$.

Let us say that $S$ is \emph{Rhodes semisimple} if $\equiv_{\GGM}$ is the trivial congruence (i.e., the equality relation). By a result of Rhodes~\cite{Rhodeschar}, this is equivalent to $S$ having a faithful semisimple (i.e., completely reducible) representation over the complex numbers.


By a \emph{partial $S$-set}, we shall mean a finite set $\Omega$ with an action of $S$ on the right by partial mappings.  We shall say \emph{$S$-set} to mean a right action of $S$ by total mappings.  Note that partial $S$-sets can be identified with pointed $S$-sets, that is, $S$-sets with a distinguished point fixed by $S$; this point is often called a sink. The \emph{strong orbit} of $\alpha\in \Omega$ is the set $\mathscr O_{\alpha}=\{\beta\in \Omega\mid \alpha S^1=\beta S^1\}$.  We say the strong orbit is \emph{null} if $\mathscr O_{\alpha}S\cap \mathscr O_{\alpha}=\emptyset$ and is \emph{transitive} otherwise.  Note that $\mathscr O_{\alpha}$ is null if and only if $\mathscr O_{\alpha}=\{\alpha\}$ and $\alpha\notin \alpha S$.

Any strong orbit can be made into a partial $S$-set by restriction of the action, that is, if $\mathscr O$ is a strong orbit and $\beta\in \mathscr O$, we define \[\beta\circ s =\begin{cases}\beta s, & \text{if}\ \beta s\in \mathscr O, \\ \emptyset, & \text{else}\end{cases}\] for $s\in S$.   We shall say that $\Omega$ is \emph{transitive} if it consists of a single transitive strong orbit.  We shall say that $\Omega$ is \emph{semisimple} if it is a coproduct (disjoint union) of transitive actions, that is, each of its strong orbits is transitive and $S$-invariant (meaning, if $\alpha\in \mathscr O$ and $\alpha s$ is defined, then $\alpha s\in \mathscr O$).  For groups, every action is semisimple and the strong orbits are just the usual orbits.  Any action of an inverse semigroup by partial bijections is semisimple and the strong orbits are the usual orbits, but actions of inverse semigroups by arbitrary partial functions need not be semisimple.

If $\Omega$ is any partial $S$-set, we define its \emph{semisimplification} to be the coproduct $\Omega_{ss}$ of all the transitive strong orbits of $\Omega$ viewed as partial $S$-sets in their own right.  So as a set $\Omega_{ss} = \{\alpha\in \Omega\mid \alpha\in \alpha S\}$ but the action is given by
\[\alpha\circ s =\begin{cases}\alpha s, & \text{if}\ \alpha s\in \mathscr O_{\alpha}, \\ \emptyset, & \text{else.}\end{cases}\]
Note that $|\Omega_{ss}|\leq |\Omega|$.

Let us fix once and for all an idempotent $e_J$ from each regular $\mathscr J$-class $J$ and denote by $G_J$ the maximal subgroup at $e_J$ and by $R_J$ the $\mathscr R$-class of $e_J$.  Note that $G_J$ acts freely (i.e., with trivial point stabilizers) on the left of $R_J$. Furthermore $R_J$ is a transitive partial $S$-set called the \emph{right Sch\"utzenberger representation} of $S$ on $R_J$. It is known that up to isomorphism $R_J$ does not depend on the chosen $\mathscr{R}$-class of $J$.

When dealing with partial transformations, we shall often write that two quantities are equal where some of these quantities may be undefined.  The equality should always be interpreted as stating that either both sides are undefined or both sides are defined and equal.   Let us say that $s\in S$ \emph{annihilates} $\alpha\in \Omega$ if $\alpha s=\emptyset$ and it annihilates  $\Omega$ if $\Omega s=\emptyset$.  Clearly the annihilator of any element of $\Omega$ is a right ideal and the annihilator of $\Omega$ is a $2$-sided ideal. The following is well known, and essentially due to Rhodes, cf.~\cite{resultsonfinite}.  We include a proof for completeness.

\begin{Lemma}\label{l:apex}
Let $\Omega$  be a transitive partial $S$-set.  Then there is a unique minimal $\mathscr J$-class $J$ with $\Omega J\neq \emptyset$.  Moreover, $J$ is regular.  There is an $S$-equivariant surjection $\p\colon R_J\to \Omega$ (where equivariant means $\p(rs)=\p(r)s$ for all $r\in R_J$ and $s\in S$, with equality understood as above).  Consequently, $s\equiv_{\RM,J}t$ implies that $s,t$ act the same on $\Omega$.
\end{Lemma}
\begin{proof}
Choose $J$ minimal such that $J$ does not annihilate $\Omega$.
Let $I=S^1JS^1\setminus J$.  Then $I$ annihilates $\Omega$.  Note that since $\Omega$ is transitive, it has no nonempty proper $S$-invariant subsets. Since it is not annihilated by $J$, we have that $\Omega = \Omega S^1JS^1 = \Omega I\cup \Omega J=\Omega J$.  It follows that $\Omega J^2=\Omega$.  Therefore, $J^2\nsubseteq I$, and so $J^2\cap J\neq \emptyset$.  But then $J$ is regular by a standard fact about finite or stable semigroups, cf.~\cite[Appendix~A]{qtheor}. If $J'\neq J$ is another such minimal $\mathscr J$-class, then $\Omega J'=\Omega$ and $J'J\subseteq I$.  Thus $\Omega =\Omega J=\Omega J'J\subseteq \Omega I=\emptyset$.  This contradiction shows that $J$ is unique.

Since the annihilator of $\Omega$ is a two-sided ideal, no element of $J$ annihilates $\Omega$.  In particular, $\Omega e_J\neq \emptyset$ and so choose $\alpha\in \Omega e_J$.  Define $\p\colon R_J\to \Omega$ by $\p(r)=\alpha r$.  Since the annihilator of $\alpha$ is a right ideal and $e_J$ does not annihilate $\alpha$, $\alpha r\in \Omega$ is defined.  Thus $\p$ makes sense as a map.  Clearly, $\Omega = \Omega e_JS=\Omega R_J$ since $\Omega$ is transitive,  $\Omega e_JS\neq \emptyset$ and $e_JS\setminus R_J\subseteq I$.  For equivariance, if $rs\notin R_J$, then $rs\in I$ and so it annihilates $\Omega$.  Thus $\p(r)s =(\alpha r)s=\alpha (rs)=\emptyset$.  If $rs\in R_J$, then $\p(rs)$ is defined by the above and $\p(rs)=\alpha(rs)=(\alpha r)s=\p(r)s$.

Finally, suppose that $s\equiv_{\RM,J} t$ and let $\alpha\in \Omega$.  Then $\alpha=\p(r)$ for some $r\in R_J$ and hence $\alpha s =\p(r)s=\p(rs)=\p(rt)=\p(r)t=\alpha t$ where, as usual, either all expressions in this string of equalities are undefined, or they are all defined and equal.  This completes the proof.
\end{proof}

The $\mathscr J$-class $J$ appearing in the lemma is called the \emph{apex} of $\Omega$.

Our next proposition will allow us to reduce to considering semisimple actions for Rhodes semisimple semigroups.

\begin{Prop}\label{p:gotosemisimple}
Let $S$ be a semigroup such that $\equiv_{\RM}$ is the equality relation (e.g., a Rhodes semisimple semigroup).  If the action of $S$ on $\Omega$ is faithful, then so is the action on $\Omega_{ss}$.  Moreover, $|\Omega_{ss}|\leq |\Omega|$.
\end{Prop}
\begin{proof}
Suppose that $s\neq t$ in $S$.  Then we can find a regular $\mathscr J$-class $J$ with $s\not\equiv_{\RM,J} t$.  Then we can find $x\in J$ with $xs\in J$ and $xs\neq xt$ or vice versa.  Assume the former for convenience and let $e$ be an idempotent in the $\mathscr R$-class of $x$.  Choose $\alpha\in \Omega$ with $\alpha xs\neq \alpha xt$.  Then since $ex=x$, we have that $\alpha exs=\alpha xs\neq \alpha xt=\alpha ext$ and so we may assume without loss of generality that $\alpha e=\alpha$. Since $S$ is stable, we must have $xs\mathrel{\mathscr R} x\mathrel{\mathscr R} e$ and so we can find $u\in S^1$ with $xsu=e$.  Then $\alpha xsu=\alpha e=\alpha$ and so $\alpha, \alpha x,\alpha xs$ are all in the same transitive strong orbit $\mathscr O$.  Then either $\alpha xt\notin \mathscr O$ or $\alpha xs\neq \alpha xt\in \mathscr O$.  In either case, the transitive  strong orbit $\mathscr O$ distinguishes $s,t$ in their action on $\Omega_{ss}$.  Thus $\Omega_{ss}$ is faithful.  The final statement is trivial.
\end{proof}

As a corollary, we see that a minimal degree faithful representation of a Rhodes semisimple semigroup by partial transformations can be achieved by a semisimple action.

\begin{Cor}\label{c:achieve.semisimple}
Let $S$ be a semigroup such that $\equiv_{\RM}$ is the equality relation (e.g., a Rhodes semisimple semigroup).   Then the minimal degree of a faithful action of $S$ on a set by partial transformations is realized by a semisimple action.
\end{Cor}

Note that the two-element null semigroup $\{0,a\}$ with $a^2=0=0a=a0$ does not have a faithful semisimple action.  Its unique transitive action is the trivial action on a one-point set.  In fact, by Lemma~\ref{l:apex} if $S$ has a faithful semisimple action, then $\equiv_{\RM}$ must be the equality relation.  Conversely, if $\equiv_{\RM}$ is the equality relation, then the coproduct of the right Sch\"utzenberger representations of $S$ on regular $\mathscr R$-classes is a faithful semisimple action.

We shall now focus our attention on the case of semisimple actions and give a criterion for faithfulness.  The idea follows that used by the second author in~\cite{GZR} for matrix representations and generalizes the approaches of Easdown~\cite{Easdownfaithful} for representations of semilattices of groups by partial transformations and Schein for inverse semigroups~\cite{scheininvmindeg}.

Let us say that a regular $\mathscr J$-class $J$ is \emph{$\RM$-reducible} if ${\bigcap_{J'<J}\equiv_{\RM,J'}}\subseteq {\equiv_{\RM,J}}$.  Otherwise, we say that $J$ is \emph{$\RM$-irreducible}.  Note that $J$ is $\RM$-irreducible if and only if there exist $s,t\in S$ with $s\not\equiv_{\RM,J} t$ and $s\equiv_{\RM,J'} t$ for all $J'<J$.  From this it is immediate that the intersection defining $\equiv_{\RM}$ can be restricted to $\RM$-irreducible $\mathscr J$-classes.

\begin{Prop}\label{p:irredundant}
The congruence $\equiv_{\RM}$ is the intersection of the $\equiv_{\RM,J}$ as $J$ runs over the $\RM$-irreducible regular $\mathscr J$-classes.  Moreover, no $\RM$-irreducible regular $\J$-class $J$ can be omitted from this intersection without changing the congruence.
\end{Prop}
\begin{proof}
To establish the equality of congruences, we just need that if $s\not\equiv_{\RM} t$, then $s\not\equiv_{\RM,J} t$ for some $\RM$-irreducible $\mathscr J$-class $J$.  Choose $J$ minimal with $s\not\equiv_{\RM,J} t$.  Then $s\equiv_{\RM, J'} t$ for all $J'<J$ by choice of $J$, and hence $J$ is $\RM$-irreducible.   This proves the first statement.

Suppose now that $J$ is $\RM$-irreducible.  We show that there are elements $s,t\in S$ with $s\not\equiv_{\RM,J} t$ and with $s\equiv_{\RM,J'} t$ for all $J\neq J'$. Indeed, by definition of $\RM$-irreducibility, we can find $s,t\in S$ with $s\not\equiv_{\RM,J} t$ and $s\equiv_{\RM, J'} t$ for all $J'<J$.  Then we can find $x\in J$ with $xs\in J$ and $xs\neq xt$ (or vice versa, but we can assume this case without loss of generality).  Then $xs\equiv_{\RM,J'} xt$ for all $J'<J$ and $xs\not\equiv_{\RM,J} xt$ because if $e$ is an idempotent $\mathscr R$-equivalent to $x$, then $exs=xs\in J$ and $exs=xs\neq xt=ext$.  But since $xs,xt\in SJS$, if $J'\nleq J$, then $xs\equiv_{\RM, J'} xt$ because $J'xs\cap J'=\emptyset=J'xt\cap J'$.  Thus $xs\equiv_{\RM,J'} xt$ for every $J'\neq J$.  This shows that $\equiv_{\RM,J}$ cannot be omitted from the intersection.
\end{proof}

Thus the $\RM$-irreducible $\J$-classes provide the unique minimal way to express $\equiv_{\RM}$ as an intersection of congruences of the form $\equiv_{\RM,J}$.

If $J$ is an $\RM$-irreducible $\mathscr J$-class, put \[M_J=\{g\in G_J\mid g\equiv_{\RM,J'} e_J, \forall J'<J\}.\]  Note that $M_J$ is a normal subgroup of $G_J$.

If $\Omega$ affords a semisimple action of $S$ and $J$ is a regular $\mathscr J$-class, then $\Omega_J$ will denote the union of all strong orbits of $\Omega$ with apex $J$.  Note that $\Omega_J$ is $S$-invariant.  Also note that $\Omega_Je_J$ is a $G_J$-set.

\begin{Thm}\label{t:faithful.crit}
Let $S$ be a Rhodes semisimple semigroup with a semisimple action on a finite set $\Omega$.  Then the action is faithful if and only if, for each $\RM$-irreducible $\mathscr J$-class $J$:
\begin{enumerate}
\item $\Omega_J\neq \emptyset$;
\item $\Omega_Je_J$ is a faithful $M_J$-set.
\end{enumerate}
\end{Thm}
\begin{proof}
We begin with an observation.  If $J$ is a regular $\J$-class and $s,t\in SJS$ with $s\equiv_{\RM,J'} t$ for all $J'<J$, then $s,t$ act the same on any strong orbit $\mathscr O$ with apex different from $J$.  Indeed, if $\mathscr O$ has apex $J'$ with $J'\nleq J$, then $s,t$ both annihilate $\mathscr O$.  On the other hand, if $J'<J$, then $s,t$ act the same on $\mathscr O$ by Lemma~\ref{l:apex}.

Suppose now that $\Omega$ is faithful.  If $J$ is an $\RM$-irreducible $\mathscr J$-class, then we can find $s,t\in S$ with $s\not\equiv_{\RM, J} t$ and $s\equiv_{\RM,J'} t$ for all $J'<J$.  Then we can find $x\in J$ with $xs\in J$ and $xs\neq xt$ or vice versa.  Note that $xs\equiv_{\RM,J'}xt$ for all $J'<J$ and $xs,xt\in SJS$.  Therefore, $xs,xt$ act the same on any strong orbit with apex different from $J$ by the observation.  Since $\Omega$ is faithful and $xs\neq xt$, we deduce that $\Omega_J\neq \emptyset$ as it must distinguish these elements.  This establishes (1).

If $g\in M_J$ with $g\neq e_J$, then since $g\equiv_{\RM,J'} e_J$ for all $J'<J$ and $g,e_J\in SJS$, it follows from the observation that $g,e_J$ act the same on all strong orbits not belonging to $\Omega_J$.  Hence they act differently on $\Omega_J$.  But if $\alpha\in \Omega_J$ with $\alpha g\neq \alpha e_J$, then $(\alpha e_J)g=\alpha(e_Jg)=\alpha g\neq \alpha e_J$, and so $g$ acts nontrivially on $\Omega e_J$, yielding (2).

Assume now that (1) and (2) hold.  If $s\neq t$ in $S$, then since $S$ is Rhodes semisimple and ${\equiv_{\RM}}\subseteq {\equiv_{\GGM}}$, we have that $\equiv_{\RM}$ is the equality relation.  Thus we can find a regular $\J$-class  $J$ with $s\not\equiv_{\RM,J} t$ and we may assume that $J$ is minimal with this property.  But then $s\equiv_{\RM, J'} t$ for all $J'<J$ and so $J$ is $\RM$-irreducible.  There exists $x\in J$ with $xs\in J$ and $xs\neq xt$ or vice versa.  Without loss of generality we may assume that $xs\in J$.   Since $S$ is Rhodes semisimple, we can find a regular $\mathscr J$-class $J_0$ with $xs\not\equiv_{J_0} xt$.  Since ${\equiv_{\RM, J_0}}\subseteq {\equiv_{J_0}}$ and $xs\equiv_{\RM,J'} xt$ for all $J'<J$, we cannot have $J_0<J$.  But since $xs,xt\in SJS$, we also have $xs\equiv_{J'} xt$ whenever $J'\nleq J$ as $uxsv,uxtv\notin J'$ for any $u,v\in J'$.  Thus we must have that $xs\not\equiv_J xt$.  Hence we can find $u,v\in J$ with $uxsv\in J$ and $uxsv\neq uxtv$ (or vice versa).  Resetting notation we can assume that there are $a,b\in J$ with $asb\in J$ and $asb\neq atb$ (putting $a=ux$, $b=v$ and perhaps interchanging $s$ and $t$).  We shall also use that $asb\equiv_{\RM,J'} atb$ for all $J'<J$.

If $atb\notin J$, then since $atb\in SJS$, we deduce that $\Omega_J asb\neq \emptyset$ and $\Omega_J atb=\emptyset$.  Thus $asb, atb$ act differently on $\Omega$, whence $s,t$ act differently on $\Omega$.  So we may assume that $asb, atb\in J$ and $asb\neq atb$.  Then $asb\HH atb$ by stability, and so, by Green's lemma, we can find $\ell,r\in J$ so that $\theta\colon H_{asb}\to G_J$ given by $\theta(x) = \ell xr$ is a bijection with $\theta(asb)=e_J$.  Then $g=\theta(atb)\neq e_J$ and $g=\ell atbr \equiv_{\RM,J'} \ell asbr =e_J$ for all $J'<J$.  Thus $g\in M_J$ and so by assumption, there is $\alpha\in \Omega_Je_J$ with $\alpha g\neq \alpha$.  Therefore, $\alpha \ell atbr \neq \alpha =\alpha e_J = \alpha \ell asbr$, and so $s,t$ act differently on $\Omega$.  Thus $\Omega$ is faithful.
\end{proof}

A \emph{congruence} on a partial $S$-set $\Omega$ is an equivalence relation $\sim$ such that if $\alpha\sim \beta$ and $s\in S$, then $\alpha s\sim \beta s$ in the sense that either both sides are undefined or both sides are defined and equivalent.  Then $\Omega/{\sim}$ becomes a partial $S$-set via $[\alpha]s=[\alpha s]$, where $[x]$ denotes the $\sim$ class of $x$, which is well defined (and where $[\emptyset]$ is interpreted as $\emptyset$).  If $f\colon \Omega\to \Lambda$ is a surjective $S$-equivariant map of partial $S$-sets, then $\Lambda\cong \Omega/{\sim}$ for the congruence given by $\alpha\sim \beta$ if $f(\alpha)=f(\beta)$.

If $\Omega$ is any partial $S$-set and $e\in S$ is an idempotent, then $\Omega e$ is a partial $eSe$-set with $e$ acting as the identity; in particular, it is a $G_e$-set for the maximal subgroup $G_e$ at $e$.  In~\cite{transformations} the second author showed that there is a unique largest congruence on $\Omega$ whose restriction to $\Omega e$ is the equality relation.  We call it the \emph{Green's congruence} associated to $e$ because it is the congruence analogue of Green's submodule of a module over a ring associated to an idempotent in~\cite{Greenpoly}.  We shall call $\Omega/{\sim}$ the \emph{Green's quotient} of $\Omega$ with respect to $e$.

\begin{Prop}\label{p:greens.cong}
If $\Omega$ is a partial $S$-set and $e\in S$ is an idempotent, define $\sim$ by $\alpha\sim \beta$ if $\alpha se=\beta se$ for all $s\in S$ (with our usual understanding of equality).  Then $\sim$ is the largest congruence such that $\sim$ is the equality relation on $\Omega e$.
\end{Prop}
\begin{proof}
Trivially, if $\alpha\sim \beta$ and $t\in S$, then $(\alpha t)se=\alpha (ts)e =\beta (ts)e=(\beta t)se$ (with our usual understanding of equality), and so $\alpha t\sim \beta t$ (with our usual understanding). Also note that if $\alpha,\beta\in \Omega e$ and $\alpha\sim\beta$, then $\alpha =\alpha ee=\beta ee=\beta$, and so $\sim$ restricts to equality on $\Omega e$.   Suppose that $\equiv$ is a congruence which restricts to equality on $\Omega e$.  Then if $\alpha\equiv \beta$ and $s\in S$, then $\alpha se\equiv \beta se$, and so $\alpha se=\beta se$ (with our usual understanding of equality) by assumption on $\equiv$.  Thus ${\equiv}\subseteq {\sim}$.
\end{proof}

We recall some basic facts about tensor products. We restrict to the case when one of the factors is a right Sch\"utzenberger representation $R_J$ and the tensor product is over a group. Let $X$ be a right $G_J$-set.  Then the tensor product $X\otimes_{G_J}R_J$ is the orbit space $(X\times R_J)/G_J$ where $G_J$ acts on the right of $(X\times R_J)$ via $(x,r)g = (xg,g\inv r)$.  We write $x\otimes r$ for the orbit of $(x,r)$ and note that $xg\otimes r=x\otimes gr$. Since $R_J$ is a partial $G_{J}$-$S$ biset, $X\otimes_{G_J} R_J$ is a partial $S$-set via $(x\otimes r)s= x\otimes rs$ where as usual, if $rs$ is undefined we interpret $(x\otimes r)s$ as undefined.  This does not depend on the choices as $S$ acts on the right of $X\times R_J$ by acting trivially on $X$ and in its usual way on $R_J$ (with the action defined only when it is defined on both coordinates) and the $G_J$-action is by automorphisms of this partial $S$-set.   Since the action of $G_J$ on the left of $R_J$ is free, if $T$ is a transversal for the set of $G_J$-orbits on $R_J$ (i.e., $T$ is a set of $\HH$-class representatives of the $\HH$-classes in $R_J$) with $e_J\in T$, then every element of $X\otimes_{G_J} R_J$ can be written uniquely in the form $x\otimes t$ with $t\in T$.  In particular, $|X\otimes_{G_J} R_J|=|X|\cdot |T|$ and $|T|$ is the number of $\mathscr L$-classes in $J$.  Note that $X\cong (X\otimes_{G_J} R_J)e_J$ as a $G_J$-set via $x\mapsto x\otimes e_J$.

\begin{Prop}\label{p:quotient.of.tensor}
Let $\Omega$ be a semisimple partial $S$-set such that each strong orbit has apex $J$.  Then there is a surjective $S$-equivariant mapping \[\theta\colon \Omega e_J\otimes_{G_J} R_J\to \Omega,\] which is injective on $(\Omega e_J\otimes_{G_J} R_J)e_J$.  Therefore, the Green's quotient $(\Omega e_J\otimes_{G_J} R_J)/{\sim}$ is a quotient of $\Omega$ satisfying $[(\Omega e_J\otimes_{G_J} R_J)/{\sim}]e_J\cong \Omega e_J$ as $G_J$-sets.
\end{Prop}
\begin{proof}
If $\alpha\in \Omega e_J$, the proof of Lemma~\ref{l:apex} shows there is a surjective $S$-equivariant map $\p_{\alpha}\colon R_J\to \mathscr O_{\alpha}$ given by $\p_{\alpha}(r)=\alpha r$.  Hence we get a surjective $S$-equivariant map $\Phi\colon \Omega e_J\times R_J\to \Omega$ (with the action of $S$ on $\Omega e_J$ trivial) given by $(\alpha,r)\mapsto \alpha r$ by viewing  $\Omega e_J\times R_J$ as a coproduct of $|\Omega e_J|$ copies of $R_J$.  But if $g\in G_J$, then $\Phi(\alpha g,g\inv r) = \alpha gg\inv r=\alpha e_Jr=\alpha r$ and so $\Phi$ factors through a surjective $S$-equivariant map $\theta\colon \Omega e_J\otimes_{G_J} R_J\to \Omega$ with $\theta(\alpha\otimes r) = \alpha r$.  In particular, if $\alpha \in \Omega e_J$, then  $\theta(\alpha\otimes e_J)= \alpha e_J=\alpha$ and so $\theta$ is injective on $(\Omega e_J\otimes_{G_J} R_J)e_J$.  Thus $(\Omega e_J\otimes_{G_J} R_J)/{\sim}$ is a quotient of $\Omega$ by Proposition~\ref{p:greens.cong}.  The final statement follows from Proposition~\ref{p:greens.cong} since $(\Omega e_J\otimes_{G_J} R_j)e_J\cong \Omega e_J$ as a $G_J$-set.
\end{proof}

Suppose now that $S$ is Rhodes semisimple.  Theorem~\ref{t:faithful.crit} implies that if $\Omega$ is a faithful semisimple partial $S$-set of minimal degree, then $\Omega_J=\emptyset$ whenever $J$ is $\RM$-reducible.  Proposition~\ref{p:quotient.of.tensor} shows that $(\Omega_J e_J\otimes_{G_J} R_J)/{\sim}$ is a quotient of $\Omega_J$ with $[(\Omega_J e_J\otimes_{G_J} R_J)/{\sim}]e_J\cong \Omega_Je_J$ as a $G_J$-set and hence  replacing each $\Omega_J$ by $(\Omega_J e_J\otimes_{G_J} R_J)/{\sim}$ for $J$ an $\RM$-irreducible $\J$-class will result in a faithful partial $S$-set by Theorem~\ref{t:faithful.crit}.  By minimality it follows that we must in fact have $\Omega_J\cong (\Omega_J e_J\otimes_{G_J} R_J)/{\sim}$.   Thus a minimal degree faithful action of $S$ can be obtained by finding, for each $\RM$-irreducible $\J$-class $J$, a $G_J$-set $X_J$ which is faithful on $M_J$ and which minimizes the cardinality of $(X_J\otimes_{G_J} R_J)/{\sim}$.  Note that if $X_J$ has any isomorphic orbits then since the tensor product and Green's congruence commute with coproducts, we may remove repeated orbits.  Hence our search space is limited to a finite number of $G_J$-sets as we need at most one coset space for each conjugacy class of subgroups.

So define $d_J$ for an $\RM$-irreducible $\J$-class $J$ to be the minimum cardinality of a quotient $(X_J\otimes_{G_J}R_J)/{\sim}$ where $X_J$ runs over all permutation representations of $G_J$ that are faithful when restricted to $M_J$ and have no two isomorphic orbits (i.e., stabilizers of different orbits are non-conjugate).  Then the above argument proves the following.

\begin{Thm}\label{t:min.degree}
Let $S$ be a Rhodes semisimple semigroup.  Then the minimal degree of a faithful representation of $S$ by partial transformations is $\sum d_J$ where the sum runs over the $\RM$-irreducible regular $\J$-classes.
\end{Thm}

Of course, Theorem~\ref{t:min.degree} restricts to groups as the obvious fact  that minimal degree of a faithful representation by partial mappings is the same as the minimal degree by permutations.

There is one special case where we can say more about $d_J$.

\begin{Prop}\label{p:llm.case}
Let $J$ be a regular $\J$-class and suppose that for any two distinct $\mathscr L$-classes $L,L'$ of $J$, there is an $R$-class $R$ of $J$ such that exactly one of $R\cap L$ and $R\cap L'$ contains an idempotent.  Then, for any $G_J$-set $X_J$, the Green's congruence on $X_J\otimes_{G_J} R_J$ is the equality relation.
\end{Prop}
\begin{proof}
Choose a transversal $T$ to the $\HH$-classes of $R_J$ containing $e_J$. Then each element of $X_J\otimes_{G_J} R_J$ has a unique representative of the form $x\otimes t$ with $X\in X_J$ and $t\in T$.  Assume that $x\otimes t\neq x'\otimes t'$.  

Suppose first that $t\neq t'$ and let $L, L'$ be their respective $\mathscr L$-classes.  Then $L\neq L'$, and so, by assumption, there is an $\mathscr R$-class $R$ such that $R\cap L$ contains an idempotent $e$ and $R\cap L'$ does not contain an idempotent (up to a change of notation).  Assume that $ty=e_J$ with $y\in S^1$.  Then $teye_J=tye_J=e_J$ and so $(x\otimes t)eye_J = x\otimes e_J$.  On the other hand, because $R\cap L'$ does not contain an idempotent, $t'e\notin J$ by the Clifford-Miller theorem (cf.~\cite[Theorem~2.17]{CP}).  Thus $(x'\otimes t')eye_J=\emptyset$ and so $x\otimes t\nsim x'\otimes t'$.  

If $t=t'$ and $ty=e_J$, then $(x\otimes t)ye_J=x\otimes e_J$ and $(x'\otimes t)ye_J = x'\otimes e_J$, and so if $x\neq x'$, then $x\otimes t\nsim x'\otimes t'$. This completes the proof.
\end{proof}

\begin{Thm}\label{t:llm.case.nice}
Let $S$ be a Rhodes semisimple semigroup and assume that each $\RM$-irreducible $\J$-class satisfies the condition that, for any two distinct $\mathscr L$-classes $L,L'$ of $J$, there is an $R$-class $R$ of $J$ such that exactly one of $R\cap L$ and $R\cap L'$ contains an idempotent.  Then the minimal degree of a faithful action of $S$ by partial transformations is $\sum \ell_J\cdot n_J$ where the sum runs over the $\RM$-irreducible $\J$-classes $J$ of $S$, and  where $\ell_J$ is the number of $\mathscr L$-classes in $J$ and $n_J$ is the smallest degree of a permutation representation of $G_J$ whose restriction to $M_J$ is faithful.
\end{Thm}

We note that, for a regular $\J$-class $J$, the property that, for any two distinct $\mathscr L$-classes $L,L'$ of $J$, there is an $R$-class $R$ of $J$ such that exactly one of $R\cap L$ and $R\cap L'$ contains an idempotent implies that ${\equiv_J}={\equiv_{\RM,J}}$.  Indeed, suppose that the former property holds and that $s\equiv_J t$.  Let $x\in J$ and suppose that $xs\in J$.  Then choose an idempotent $e$ with $xs\eL e$.  Then $xse\in J$ and so $xte\in J$.  Thus $xt\in J$.  We claim that $xs\eL xt$.  Indeed, let $R$ be an $\R$-class and let $y\in R$.  Then $xsy\in J$ if and only if $xty\in J$ because $s\equiv_J t$.  Thus, by the Clifford-Miller Theorem, either both $R\cap L_{xs}$ and $R\cap L_{xt}$ contain an idempotent or neither do.  It follows from our hypothesis that $L_{xs}=L_{xt}$.  Choose an idempotent $e\in L_{xs}=L_{xt}$.  Then $xs=xse=xte=xt$ because $s\equiv_J t$.  We conclude that $s\equiv_{\RM,J} t$, and so ${\equiv_J}={\equiv_{\RM,J}}$.

It is convenient to express the condition that for any two distinct $\mathscr L$-classes $L,L'$ of $J$, there is an $R$-class $R$ of $J$ such that exactly one of $R\cap L$ and $R\cap L'$ contains an idempotent in terms of a Rees coordinitization of the principal factor $J^{0}$ corresponding to $J$. If $J^{0} \approx \mathscr M^{0}(G,A,B,C)$ then the condition is equivalent to having, for any two distinct rows $C_{b\ast}$, $C_{b'\ast}$, with $b,b'\in B$, of the sandwich matrix, that there is an $a \in A$ such that exactly one of $C_{ba},C_{b'a}$ is non-zero. Since the sandwich matrices of principal factors of inverse semigroups are identity matrices (or equivalently because each $\mathscr L$- and $\mathscr R$-class of an inverse semigroup has a unique idempotent), the hypotheses of Theorem~\ref{t:llm.case.nice} apply, in particular, to inverse semigroups.  For inverse semigroups, the action on $X_J\otimes_{G_J}R_J$ will be by partial bijections and so the minimal degrees of a faithful representation by partial bijections and by partial transformations coincide.

Easdown began the first serious consideration of the minimal faithful degree of inverse semigroups, beginning with semilattices of groups~\cite{Easdownfaithful}.  The general case was first handled by Schein~\cite{scheininvmindeg}.  Theorem~\ref{t:llm.case.nice} gives the following computation of the minimal faithful degree.

\begin{Cor}\label{c:inverse}
Let $S$ be an inverse semigroup.  Then the minimal degree of a faithful representation of $S$ by partial transformations coincides with the minimal degree of a faithful representation by partial bijections.  This degree is $\sum_J \ell_J\cdot n_J$ where $J$ runs over the $\RM$-irreducible $\J$-classes, $\ell_J$ is the number of $\mathscr L$-classes in $J$ and $n_J$ is the minimal degree of a permutation representation of $G_J$ that restricts faithfully to $M_J$, where $M_J$ is the normal subgroup of the maximal subgroup $G_J$ consisting of those elements that act trivially under all Sch\"utzenberger representations associated to $\J$-classes $J'<J$.
\end{Cor}

Schein~\cite{scheininvmindeg} defines a $\mathscr J$-class to be irreducible if all its elements are join irreducible in the natural partial order on $S$.  He shows that this is equivalent to even a single element being join irreducible. Note that $E(S)$ is a meet semilattice and contains a minimum element.  It is straightforward to see that if $e\in E(S)$ is not the minimum element, then $e$ is join irreducible in $S$ if and only if it is join irreducible in $E(S)$, but the minimum element is join irreducible in $S$ if and only if it is not a zero for $S$.  The next proposition and a quick read of Schein's notation shows that Corollary~\ref{c:inverse} is equivalent to Schein's theorem.

\begin{Prop}\label{p:Schein.irred}
Let $S$ be an inverse semigroup.  Then a $\J$-class $J$ of $S$ is $\RM$-irreducible if and only if all its elements are join irreducible in the natural partial order on $S$.  Moreover, if $e_J\in J$ is an idempotent, then $M_J = \{g\in G_J\mid f<e_J\implies f<g\}$.
\end{Prop}
\begin{proof}
First note that for any $\J$-class $J'$, we have that $s\equiv_{\RM,J'} t$ if and only if for all idempotents $e\in J'$, we have $es\in J'\iff et\in J'$ and if this occurs then $es=et$ (since each element of $J'$ is $\mathscr L$-equivalent to an idempotent).  Also, we claim that if $f\in S$ is any idempotent and $s\in S$, then $fs\in J_f$ if and only if $f\leq ss^{-1}$.  Indeed, if $f\leq ss^{-1}$, then $fs(fs)^{-1} = fss^{-1}=f$ and so $fs\in J_f$.  Conversely, if $fs\in J_f$, then $fs\mathrel{\mathscr R} f$ by stability, and so $f=fs(fs)^{-1} = fss^{-1}$ establishing $f\leq ss^{-1}$.

 Suppose first that $J$ is an $\RM$-reducible $\J$-class and let $e\in J$ be an idempotent.  If $J$ is the minimal ideal of $S$, then the fact that $J$ is $\RM$-reducible means that $\equiv_{\RM,J}$ is the universal relation.  In particular, $e\equiv_{\RM,J} s$ for all $s\in S$, and so $e=ee=es$ as $e\in E$.  Thus $e$ is a zero element (as $se=(es^{-1})\inv = e$). Therefore, $J=\{e\}$ and $e$ is a join reducible element of $S$ in the natural partial order (being the join of the empty set).  So suppose that $J$ is not the minimal ideal, and hence $e$ is not the minimum idempotent of $S$.   We show that $e$ is the join of all idempotents strictly below it, i.e., is join reducible.  Let $f$ be an idempotent such that, for all $f'<e$, we have $f'\leq f$.  We want to show that $e\leq f$ or, equivalently, that $ef=e$.  Let $X=\{J'\mid J'<J\}$.  We claim that $e\equiv_{\RM, J'} ef$ for all $J'\in X$.    Else let $J'\in X$.  Let $x\in J'$ be an idempotent.  Then if $xef\in J$, trivially $x\geq xe\geq xef\in J$, and so $xe\in J$.  Conversely, if $xe\in J$, then by stability $xe\mathrel{\mathscr{R}} x$, and so $x=xe$ since $\mathscr R$-classes have a unique idempotent.   Therefore $x\leq e$. But since $J'<J$, in fact $x<e$, and so $x\leq f$. Thus $xef=xfe=xe$, whence $ef\equiv_{\RM,J'} e$.  Since this is true for all $J'<J$, we have by reducibility that $ef\equiv_{\RM,J} e$.  Since $ee=e\in J$, we deduce that $e=eef=ef$, whence $e\leq f$, as required.

Conversely, suppose that each idempotent of the $\J$-class $J$ is join reducible.  If $J$ is the minimal ideal, then since the minimum idempotent is join reducible, it must be the join of the empty set and hence is the only minimal element of $S$.  Thus $J=\{e\}$ and $e$ is a zero element.  But then $\equiv_{\RM,J}$ is the universal relation and so $J$ is $\RM$-reducible.    Suppose now that $J$ is not the minimal ideal and that $s\equiv_{\RM,J'} t$ for all $J'<J$.  Then we want that $s\equiv_{\RM,J} t$.  Using the observation in the first line of the proof, let $e\in J$ and suppose that $es\notin J$.  Then, by the claim, $e\nleq ss^{-1}$ and so we can find $f<e$ with $f\nleq ss^{-1}$ as $e$ is the join of the idempotents strictly below it (being join reducible and not the minimum element of $E(S)$).  Note that $J_f<J$ as idempotents in a $\mathscr J$-class of a finite inverse semigroup are incomparable by stability.  By the claim $fs\notin J_f$ and hence $ft\notin J_f$ as $s\equiv_{\RM, J_f} t$.  But this means that $f\nleq tt\inv$ by the claim and hence $e\nleq tt\inv$ as $f<e$.  Thus $et\notin J$.  We conclude that $es\in J\iff et\in J$.  Suppose now that $es, et\in J$ and let $X=\{f\in E(S)\mid f<e\}\neq \emptyset$, as $e$ is not the minimum element of $E(S)$.
By the claim $e\leq ss\inv,tt\inv$ and hence $f\leq ss\inv,tt\inv$ for all $f\in X$.   Then $J_f<J$ for each $f\in X$ and $fs,ft\in J_f$ by the claim.  Since $s\equiv_{\RM, J_f} t$, we must have that $fs=ft$.  By~\cite[Chapter~1, Proposition 19]{Lawson} we have $es=\bigvee_{f\in X}fs=\bigvee_{f\in X}ft=et$.  This shows that $s\equiv_{\RM,J} t$, establishing that $J$ is $\RM$-reducible.  This completes the proof that Schein's notion of irreducibility coincides with $\RM$-irreducibility (it also coincides with the notion of irreducibility in~\cite{GZR}).

For the final statement, if $g\in M_J$ and $f<e_J$, then $J_f<J$ and $fe_J=f\in J$.  Thus, since $g\equiv_{\RM,J_f} e_J$, we must have that $fg=f$, that is $f\leq g$.  Of course since $g\in J>J_f$, we must have $f<g$.   Conversely, suppose that $f<e_J$ implies $f<g$ and let $J'<J$.  If $f\in J'$ is an idempotent, then by the claim $fe_J\in J'\iff f\leq e_J\iff fg\in J'$ as $gg\inv=e_J$.  Moreover, if $fe_J\in J'$, then $fe_J\mathrel{\mathscr R} f$ by stability and so $f=fe_J$. As $J'<J$, we conclude that $f<e_J$.  But then $f<g$ and so $fg=f$.  Thus $e_J\equiv_{\RM, J'} g$, and so $g\in M_J$ as required.  This completes the proof.
\end{proof}

A semigroup $S$ is called \emph{right  mapping}~\cite{KRannals,qtheor,Arbib} if it has a (distinguished) regular $\J$-class $J$ such that $\equiv_{\RM,J}$ is the equality relation.  In this case, that $J$ is either the minimal ideal, or $S$ has a zero and $J^0$ is the unique $0$-minimal ideal.  It is immediate from Proposition~\ref{p:irredundant} that the distinguished $\J$-class is the unique $\RM$-irreducible $\J$-class of $S$.  The semigroup $S$ is \emph{generalized group mapping} if it has a (distinguished) regular $\J$-class $J$ such that $\equiv_J$ is the equality relation.  Then $S$ and $S^{op}$ are right mapping with distinguished $\J$-class $J$, and the converse is true as well.  In other words, $S$ is generalized group mapping if it has a ($0$-)minimal ideal on which $S$ acts faithfully on both the left and right.  If the maximal subgroup of $J$ is nontrivial, then $S$ is called \emph{group mapping}, and otherwise it is called $\mathsf{AGGM}$, which is an acronym for aperiodic generalized group mapping, although only the ($0$-)minimal ideal need be aperiodic.

A simple corollary of Theorem~\ref{p:llm.case} is the following, which handles  $\mathsf{AGGM}$-semigroups.

\begin{Cor}\label{c:aggm}
Let $S$ be a semigroup containing a regular $\J$-class $J$ with trivial maximal subgroup such that $\equiv_J$ is the equality relation.  Then the minimal degree of a faithful action by partial transformations of $S$ is the number of $\mathscr L$-classes in $J$.
\end{Cor}
\begin{proof}
It is well known that either $S$ is trivial or $J$ is the unique minimum non-zero $\mathscr J$-class of $S$~\cite[Chapter~4]{qtheor}.  Hence $J$ is the unique $\RM$-irreducible $\J$-class (since $\equiv_{\RM,J}$ is also the equality relation).  Since the maximal subgroup of $J$ is trivial, Theorem~\ref{t:llm.case.nice} immediately yields the result.
\end{proof}

For example, if $R_n$ is the semigroup of binary relations on an $n$-element set, then it is an $AGGM$ semigroup whose 0-minimal $\J$-class
$J$-class consists of matrices of Schein rank at most 1 and  has $2^n-1$ $\mathscr L$-classes~\cite{kimbooleanmatrices}. It follows from Corollary~\ref{c:aggm} that  $2^n-1$ is the minimal degree of a faithful action of $S$ by partial transformations. It can be realized as the natural action on the non-empty subsets of $[n]$ and equivalently by the right Sch\"utzenberger representation on its 0-minimal ideal.

\begin{Rmk}
Note that if $S$ is a semigroup with a zero element, then under any minimal degree faithful representation of $S$ by partial mappings, the zero element must act as the empty map (else remove the image of zero, which is invariant).  Also, under any minimal degree faithful representation by total mappings, the zero element must be a constant map to a sink. If follows that any minimal degree faithful representation by total mappings is obtained from a minimal degree faithful representation by partial transformations by adding a sink.
\end{Rmk}

As an application we compute the minimal faithful degree of an action  by partial transformations of $M_n(F)$ with $F$ a finite field.

\begin{Cor}\label{c:finite.field}
Let $F$ be a finite field of $q$ elements. If $|F|=2$, then the minimal degree of a faithful action of $M_n(F)$  by partial mappings is $2^n-1$.  Otherwise, suppose that $q-1$ has prime factorization $p_1^{m_1}\cdots p_k^{m_k}$.   Then the minimal degree of a faithful action of $M_n(F)$  by partial mappings is $\frac{q^n-1}{q-1}\cdot \sum_{i=1}^k p_i^{m_i}$.
\end{Cor}
\begin{proof}
Let $J$ be the $\mathscr J$-class of rank $1$ matrices. We view $F^n$ as column vectors and $M_n(F)$ as acting on the left.  Note that $G_J\cong F^\times$, which is cyclic of order $q-1$.

It is well known and easy to see that $\equiv_J$ (and hence $\equiv_{\RM,J}$) is the equality relation, and so $M_n(F)$ is Rhodes semisimple.  Indeed, suppose $A\neq B\in M_n(F)$.  Then  without loss of generality,  $Bv\neq Av$ for some $v\in F^n\setminus \{0\}$.  Choose a functional $f\colon F^n\to F$ with $f(Av-Bv)\neq 0$, and let $w\in F^n$ with $w^T x=f(x)$ for all $x\in F^n$.  Then $X=[w\mid 0]^T$ and $Y=[v\mid 0]$ are rank $1$ matrices with \[X(A-B)Y = \begin{bmatrix}f(Av-Bv) & 0_{1,n-1} \\ 0_{n-1,1} & 0_{n-1}\end{bmatrix}\neq 0.\]  Thus $XAY\neq XBY$ (and, in particular, they are not both $0$), and so $A\not\equiv_J B$.

Also $J$ is the unique $\RM$-irreducible $\J$-class as a consequence of Proposition~\ref{p:irredundant}, since $\equiv_{\RM,J}$ is the equality relation.  Clearly, $M_J=G_J$ since only the zero $\J$-class is below $J$.  Finally, note that the conditions of Theorem~\ref{t:llm.case.nice} apply.  The $\mathscr L$-class of a matrix is determined by its right null space and the $\mathscr R$-class is determined by its column space. If $V$ is a one-dimensional subspace and $W$ is an $(n-1)$-dimensional subspace, then the $\mathscr H$-class consisting of rank $1$ matrices with column space $V$ and right null space $W$ contains an idempotent if and only if $V\cap W=\{0\}$.  The corresponding idempotent is the projection from $F^n\to V$ with kernel $W$ coming from the direct sum decomposition $F^n=V\oplus W$.  If we have two $\mathscr L$-classes corresponding to matrices with one-dimensional column spaces $V_1\neq V_2$, then we can choose a hyperplane $W$ containing $V_2$ but not $V_1$, and so the $\mathscr H$-class corresponding to $V_1$ and $W$ contains an idempotent, whereas that corresponding to $V_2$ and $W$ does not.

Theorem~\ref{p:llm.case} now implies the minimal degree of $M_n(F)$ is the product of the number of $\mathscr L$-classes of $J$ with the minimal degree of a faithful permutation representation of the cyclic group $F^\times$.  Note that $J$ has $(q^n-1)/(q-1)$  $\mathscr L$-classes, as these are in bijection with $(n-1)$-dimensional (or equivalently $1$-dimensional) subspaces of $F^n$.  If $q=2$, then trivially the minimal degree of $F^\times$ is $1$.    The minimal faithful degree of a permutation representation of a finite abelian group is well known~\cite{minimalperm}. For a nontrivial cyclic group, the minimal degree is the sum of the prime powers appearing in the prime factorization of its order.  This completes the proof.
\end{proof}

Suppose that $J$ is a regular $\mathscr J$-class with maximal subgroup $G$ with identity $e$ and let $X$ be a $G$-set.  Then to compute the Green's quotient $(X\otimes_{G} R_e)/{\sim}$, it turns out that we need to only understand the action of elements of the $\mathscr L$-class $L_e$ of $e$.  More precisely, we have that $x\otimes r\sim x'\otimes r'$ if and only if $(x\otimes r)t=(x'\otimes r')t$ for all $t\in L_e$.
Indeed, if $s\in S$, then there are two cases.  If $se\notin J$, then $se$ annihilates $R_e$ and hence annihilates $X\otimes_{G} R_e$.  Thus $(x\otimes r)se=(x'\otimes r')se$ for all $x,x'\in X$ and $r,r'\in R$.  Otherwise, $se\in L_e$ by stability, and so taking $t=se$  proves the claim (using the definition of $\sim$).

By Rees's theorem, we have that $J^0$ is isomorphic to a Rees matrix semigroup $\mathscr M^0(G,A,B,C)$ where $A$ is in bijection with the set of $\mathscr R$-classes of $J$ and $B$ is in bijection with the set of $\mathscr L$-classes of $J$.   We shall use $1$ for the identity of $G$ when $G$ is viewed as an abstract group for forming the Rees matrix semigroup. Let us assume that $a_0$ corresponds to the $\mathscr R$-class of $e$ and $b_0$ to the $\mathscr L$-class of $e$.  Without loss of generality, we may assume that $C_{b_0a_0}=1$, whence $e=(a_0,1,b_0)$ after identifying $J^0$ with the Rees matrix semigroup and $G$ is identified with $\{a_0\}\times G\times \{b_0\}$.   The $\mathscr R$-class $R_e$ of $e$ then becomes $\{a_0\}\times G\times B$.  We can choose the transversal  $T=\{(a_0,1,b)\mid b\in B\}$ for the $\mathscr H$-classes of $R_e$.  If $X$ is a $G$-set, then we recall that each element of $X\otimes_{G} R_e$ can be uniquely written in the form $x\otimes (a_0,1,b)$ with $b\in B$.  To simplify the notation, we shall simply write $(x,b)$ for this element.

\begin{Prop}\label{p:green.in.coord}
Retaining the notation of the previous two paragraphs, if $(a,g,b_0)\in L_e$, $x\in X$ and $b\in B$, then
 \begin{equation}\label{eq:coord.action}
 (x,b)(a,g,b_0) = \begin{cases}(xC_{ba}g,b_0), & \text{if}\ C_{ba}\neq 0\\ \emptyset, & \text{else.}\end{cases}
 \end{equation}
  Consequently, $(x,b)\sim (x',b')$ if and only if $xC_{ba}=x'C_{b'a}$ for all $a\in A$ where we interpret $y0=0$ for $y\in X$. 
\end{Prop}
\begin{proof}
Equation~\eqref{eq:coord.action} follows because
\begin{align*}
(x,b)(a,g,b_0) & =(x\otimes(a_0,1,b))(a,g,b_0) \\ &= \begin{cases}x\otimes (a_0,C_{ba}g, b_0), & \text{if}\ C_{ba}\neq 0, \\ \emptyset, & \text{else} \end{cases} \\ &= \begin{cases}(xC_{ba}g,b_0), & \text{if}\ C_{ba}\neq 0, \\ \emptyset, & \text{else.}\end{cases}
\end{align*}

The final statement is immediate from the first statement and the discussion in the previous two paragraphs.
\end{proof}

Let us consider the special case of a simple Rhodes semisimple semigroup $S$ with two $\mathscr L$-classes and two $\mathscr R$-classes. It follows from Theorem~\ref{t:proportional} that up to isomorphism $S$ is isomorphic to a Rees matrix semigroup over a group $G$ with sandwich matrix \[C=\begin{bmatrix} 1 & 1\\ 1& g\end{bmatrix}\] with $g\neq 1$.  We show that the minimal degree of a faithful action of $S$ by partial (or total) transformations is determined by how $g$ behaves under permutation representations of $G$.  We consider a slightly more general case.

\begin{Prop}\label{p:single.diagonal}
Let $G$ be a finite group, put $g_1=1$ and let $g_2,\ldots, g_n\in G$ be nontrivial elements (possibly with repetitions).
Let $S=\mathscr M(G,n,n,C)$ where  \[C=\begin{bmatrix} 1 & 1 & 1\cdots & 1\\ \vdots & g_2 & \cdots &1 \\ & & \ddots &\vdots\\ 1& \cdots & 1 &g_n  \end{bmatrix}.\]  Then both the minimal degree of a faithful action of $S$ by partial transformations and by total transformations is  \[\min\left\{n\cdot |X|-\sum_{x\in X}|\{2\leq i\leq n\mid g_i\in G_x\}|\right\}\] where $|X|$  runs over all faithful $G$-sets $X$ (or equivalently all faithful $G$-sets with no two isomorphic orbits) and $G_x$ denotes the stabilizer of $x\in X$.  In particular, if $g_2=\cdots =g_n=g\neq 1$, then  both the minimal degree of a faithful action of $S$ by partial transformations and by total transformations is  \[\min\left\{n\deg(\rho)-(n-1)|\mathrm{Fix}(\rho(g))|\right\}\] as $\rho$ runs over all faithful permutation representations of $G$ (or equivalently all faithful permutation representations with no two isomorphic orbits).
\end{Prop}
\begin{proof}
Since $S$ is simple, it is a regular $\J$-class. Put $g_0=1$.
Let $e=(1,1,1)$ and identify $R_e= \{1\}\times G\times [n]$, $L_e=[n]\times G\times \{1\}$ and $G_e=\{1\}\times G\times \{1\}$.  Note that if $s=(i,g,j)$, then $ese= (1,g,1)$, and so if $(i,g,j)\equiv_S (i',g',j')$, then $g=g'$.  Suppose that $(i,g,j)\equiv_S (i',g,j')$. If $i\neq i'$, say with $i\neq 1$, then $(1,1,i)(i,g,j)e= (1,g_ig,1)\neq (1,g,1)=(1,1,i)(i',g,j)e$, since $g_i\neq 1$.  Thus $i=i'$. A similar argument shows that $j=j'$, and so $\equiv_S$ is the equality relation, hence $S$ is Rhodes semisimple.

By Theorem~\ref{t:min.degree} (and its proof), a minimal degree faithful action by partial transformations is a Green's quotient $(X\otimes_G R_e)/{\sim}$ for some faithful $G$-set $X$.  Note that $S$ acts on $R_e$ by total mappings and hence $(X\otimes_G R_e)/{\sim}$ has an action of $S$ by total mappings.

As before, we identify $X\otimes_G R_e$ with $X\times [n]$. For any $b,b'\in [n]$, we have $C_{b1} =1=C_{b'1}$.  So by Proposition~\ref{p:green.in.coord} we have that $(x,b)\sim (x',b')$ implies $x=xC_{b1}=x'C_{b'1}=x'$.  Thus we may only identify pairs of the form $(x,i)$ and $(x,j)$ under $\sim$, and we do so if and only if $xC_{ia}=xC_{ja}$ for all $a\in [n]$.  The only two choices of $a$ for which $C_{ia}\neq C_{ja}$ is when $a=i$ or $a=j$.  In the first case, we obtain the condition $xg_i = x$ and in the second $x=xg_j$.  Thus  $(x,i)\sim (x,j)$ if and only if $g_i,g_j\in G_x$.   So the distinct elements of $(X\otimes_G R_e)/{\sim}$ are the classes $[x,1]$, with $x\in X$,  and the classes $[x,k]$ with $x\in Z$, $2\leq k\leq n$ and $g_k\notin G_x$.  The number of such classes is $n\cdot |X|-\sum_{x\in X}|\{2\leq i\leq n\mid g_i\in G_x\}|$.

This final statement follows by noting that if $g_2=\cdots =g_n=g$, then $\sum_{x\in X}|\{2\leq i\leq n\mid g_i\in G_x\}| = (n-1)|\mathrm{Fix}(\rho(g))|$ for a permutation representation $\rho$ of $G$ on the set $X$.
\end{proof}

Specializing to the case that $n=2$, we obtain the following.

\begin{Cor}\label{c:single.non.id}
Let $S=\mathscr M(G,2,2,C)$ where $G$ is a nontrivial group, $1\neq g\in G$ and \[C=\begin{bmatrix} 1 & 1\\ 1& g\end{bmatrix}.\]  Then both the minimal degree of a faithful action of $S$ by partial transformations and by total transformations is  \[\min\left\{2\deg(\rho)-|\mathrm{Fix}(\rho(g))|\right\}\] as $\rho$ runs over all faithful permutation representations of $G$ (or equivalently all faithful permutation representations with no two isomorphic orbits).
\end{Cor}
%

We remark that $2\deg(\rho)-|\mathrm{Fix}(\rho(g))|=\deg(\rho)+ |\{x\mid x\rho(g)\neq x\}|$.
A similar argument to that of Proposition~\ref{p:single.diagonal} shows that if $g_1,\ldots, g_n$ are distinct nontrivial elements of a group $G$ and \[C= \begin{bmatrix} 1& 1&\cdots & 1\\ 1 & g_1 & \cdots &g_n\end{bmatrix},\] then the minimal faithful degree of $\mathscr M(G,n,2,C)$ is the minimum of \[2\deg (\rho)- \left|\bigcap_{i=1}^n \mathrm{Fix}(\rho(g_i))\right|\] where $\rho$ runs over all faithful permutation representations of $G$.

To indicate how difficult Corollary~\ref{c:single.non.id} is to apply in practice, we consider the case that $G$ is a symmetric group.
We recall that the Rees matrix semigroups with maximal subgroup $G$ and respective sandwich matrices \[\begin{bmatrix} 1 & 1\\ 1& g\end{bmatrix}, \begin{bmatrix} 1 & 1\\ 1& h\end{bmatrix}\] are isomorphic if and only if there is an automorphism
$\alpha$ of $G$ such that $h=\alpha(g)$. This explains why the outer automorphism on $S_6$ appears in the proof of the next theorem.

\begin{Thm}\label{t:sn.case}
Let $S=\mathscr M(S_n,2,2,C)$ where \[C=\begin{bmatrix} 1 & 1\\ 1& \sigma\end{bmatrix}\] with $\sigma\in S_n$ a nontrivial permutation.  Then the minimal degree faithful action of $S$ by either partial or total functions is $2n-|\mathrm{Fix}(\sigma)|$ except when $n=6$ and $\sigma$ is a derangement that is either a $6$-cycle, a product of two disjoint $3$-cycles or a product of three disjoint $2$-cycles.  In the first case, the minimal degree is $11$, in the second case it is $9$ and in the third case it is $8$.

\end{Thm}
\begin{proof}
We apply Corollary~\ref{c:single.non.id}.  Taking as our faithful action  $\rho$ of $S_n$ the natural action yields an upper bound of $2n-|\mathrm{Fix}(\sigma)|$. In the case that $n=6$, it is well known that there is a unique nontrivial outer automorphism of $S_6$.  Moreover, any automorphism in the nontrivial class interchanges the conjugacy class of $6$-cycles with the conjugacy class of cycle type $(3,2)$,
the conjugacy class of a product of two disjoint $3$-cycles with the conjugacy class of a $3$-cycle and of transpositions with the conjugacy class of cycle type $(2,2,2)$.  All other conjugacy classes are preserved by the nontrivial outer automorphism.  Using an automorphism representing the nontrivial outer automorphism of $S_6$ yields an action of $S_6$ on $6$ points yielding the improved bound for $6$-cycles, 
products of two disjoint $3$-cycles and products of three disjoint $2$-cycles.

We must now prove that you cannot do better than the natural action in any other case.  The key observation, is that with very few exceptions, $S_n$ has no faithful transitive actions of degree strictly less than $2n-2$ except the standard one.  If  $\rho$ is a faithful permutation representation of degree $d\geq 2n-2$, then $2d-|\mathrm{Fix}(\rho(g))|\geq d+2\geq 2n$, since a nontrivial permutation moves at least two elements, and so this permutation representation won't beat the standard one.  Thus we are essentially left with checking these exceptions by hand.

All symmetric groups have a unique minimal normal subgroup, which is the alternating group $A_n$ except for $n=4$, where it is the regular Klein $4$-subgroup of $S_4$ consisting of the identity and the three derangements of order $2$.  Thus $S_n$ is subdirectly indecomposable and so it follows that any faithful action of $S_n$ contains an orbit that is faithful and hence, from the point of view of minimizing $2\deg(\rho)-|\mathrm{Fix}(\rho(\sigma))|$, we may restrict to transitive actions.  Noting that $\binom{n}{2}<2n-2$ only when $n=2, 3$, we see from~\cite[Theorem~5.2B]{dixonbook} that the only cases of transitive faithful actions of $S_n$ on fewer than $2n-2$ points are $S_n$ acting on $[n]$ in the usual way, the action of $S_6$ on $[6]$ corresponding to a nontrivial outer automorphism and the transitive action of $S_5\cong\mathrm{PGL}_2(\mathbb F_5)$ on the $6$-element projective line  over the $5$-element field $\mathbb F_5$.  We already observed that the action of $S_6$ on $6$ points coming from a nontrivial outer automorphism yields a better bound than the standard action for $6$-cycles and products of $3$-disjoint $2$-cycles and does not improve the bound for any other choice of $\sigma$.  So we are left with  the degree $6$ action of $S_5\cong \mathrm{PGL}_2(\mathbb F_5)$.
The action of $\mathrm{PGL}_2(\mathbb F_5)$ on the projective line is well known to be sharply $3$-transitive and hence any nonidentity element fixes at most $2$ elements, that is, moves at least $4$ elements.  Since $6+4= 10\geq 10-|\mathrm{Fix}(\sigma)|$, we conclude that for any $\sigma\neq 1$, the action of $S_5$ on $6$ points provide no better a bound than the standard action.  This completes the proof.
\end{proof}

We also look at the case of a simple semigroup with maximal subgroup $S_n$ and sandwich matrix \[C=\begin{bmatrix} 1 & 1\\ 1& 1\end{bmatrix}.\]

By Theorem~\ref{t:proportional} this semigroup is neither $\RM$ nor $\mathsf{LM}$ and thus not $\GGM$. It follows that it is not a Rhodes semisimple semigroup. Moreover, it is a so called rectangular group easily seen to be isomorphic to $S_{n} \times RB(2,2)$, where $RB(2,2)$ is the $2 \times 2$ rectangular band. Thus the main results of this paper do not apply to this semigroup and we shall see in the next theorem that the answer is different from Corollary~\ref{c:single.non.id}. In fact, the minimal degrees of all rectangular groups are computed in~\cite{mindegrectband}, although one requires some work to get the precise answer. The next Theorem follows from the work in that paper, but we include a proof for completeness sake. We see that in the case of a semigroup that is not Rhodes semisimple, more combinatorial arguments are needed to compute minimal degrees. See~\cite{mindegrectband} for more examples along the lines of the proof below.

\begin{Thm}
The minimal degree of $RB(2,2)$ is 4. The minimal degree of the rectangular group $S_{n} \times RB(2,2)$ is $n+2$ for $n\geq 2$.
\end{Thm}
\begin{proof}
Consider the following 4 functions in $T_4$:
\begin{gather*}
 f_{1,1}: 14\mapsto 1, 23\mapsto 2\\
f_{2,1}: 14\mapsto 1, 23\mapsto 3\\
f_{1,2}: 1\mapsto 1, 234 \mapsto 2\\
f_{2,2}: 1\mapsto 1, 234 \mapsto 3.
\end{gather*}

Direct calculations shows that $\{f_{i,j}|1 \leq i,j \leq 2\}$ form a semigroup isomorphic to $RB(2,2)$ and that $RB(2,2)$ is not isomorphic to a
subsemigroup of $T_3$. Therefore, $\mu(RB(2,2))=4$.

Now we extend each of the functions $f_{i,j}$ to functions from $\{1,\ldots, n+2\}$ to itself by making each act as the identity on $\{5,\ldots, n+2\}$.
This embeds $RB(2,2)$ as a subsemigroup of $T_{n+2}$ contained in the $\mathscr{J}$-class of elements of rank $n$. It follows that the union of the
4 $\mathscr{H}$-classes containing these 4 idempotents is a rectangular group subsemigroup of $T_{n+2}$. Since the maximal subgroup of the
$\mathscr{J}$-class of elements of rank $n$ is $S_{n}$, this is an embedding of $S_{n} \times RB(2,2)$ into $T_{n+2}$. Therefore,
 $\mu(S_{n} \times RB(2,2)) \leq n+2$.

Since the maximal subgroup of $S_{n} \times RB(2,2)$ is $S_{n}$, we have that $n \leq \mu(S_{n} \times RB(2,2)) \leq n+2$. Since the only $\mathscr{J}$-class of
$T_n$ containing $S_n$ is the group of units, it follows that $S_{n} \times RB(2,2)$ is not a subsemigroup of $T_n$.

Now consider a possible embedding of $S_{n} \times RB(2,2)$ into $T_{n+1}$. Since all elements of $S_{n} \times RB(2,2)$ are $\mathscr{J}$ equivalent, the image
of any such embedding is contained in a single $\mathscr{J}$-class of $T_{n+1}$. The only $\mathscr{J}$-classes of $T_{n+1}$ containing $S_{n}$ are the group of units and the $\mathscr{J}$-class $J_{n}$ of elements of rank $n$. Clearly there is no embedding of $S_{n} \times RB(2,2)$ into the group of units. We claim that there is no subsemigroup of $T_{n+1}$ isomorphic to $RB(2,2)$ contained in $J_{n}$ .

The $\mathscr{L}$-classes of $J_{n}$ are indexed by subsets of $\{1,\ldots, n+1\}$ of cardinality $n$. We write $\overline{i}$ for the subset $\{1, \ldots, n+1\} - \{i\}$. The
$\mathscr{R}$-classes of $J_n$ are indexed by partitions of  $\{1,\ldots, n+1\}$ into $n$ parts. Thus the partitions have $n-1$ singleton classes and a unique class of size 2. We write $\pi_{ij}$ for the partition whose unique non-singleton class is $\{i,j\}$ for $1 \leq i < j \leq n+1$. The $\mathscr{R}$-class indexed by
$\pi_{ij}$ has precisely $2$ idempotents $e_{ij}$ and $f_{ij}$: $e_{ij}$ sends $\{i,j\}$ to $i$ and the singleton classes to themselves and thus has image $\overline{j}$; and $f_{ij}$ sends $\{i,j\}$ to $j$ and the singleton classes to themselves and thus has image $\overline{i}$. Now consider a distinct $\mathscr{R}$-class
indexed by $\pi_{kl}$. Without loss of generality we can assume that $i$ does not belong to $\{k,l\}$. Then the $\mathscr{H}$-class indexed by
$(\pi_{ij},\overline{i})$ has the idempotent $f _{ij}$, but the $\mathscr{L}$-equivalent $\mathscr{H}$-class indexed by $(\pi_{kl},\overline{i})$ has no idempotent
since $\{k,l\} \subseteq \overline{i}$ as $i$ does not belong to $\{k,l\}$.  Therefore there is no subsemigroup of $J_n$ isomorphic to $RB(2,2)$. It follows that
$\mu(S_{n} \times RB(2,2) \geq n+2$ and by the above, we get $\mu(S_{n} \times RB(2,2)) = n+2$ as claimed.
\end{proof}

A slight modification of this proof shows that $\mu(G \times RB(2,2))=\mu(G)+2$ for any non-trivial group $G$. We note that the minimal embedding obtained in the previous theorem does not come from  a semisimple action of $S_{n} \times RB(2,2)$, and, in fact, $S_n\times RB(2,2)$ has no faithful semisimple action as any two $\mathscr L$-equivalent idempotents are identified under $\equiv_{\RM}$. This contrasts to the case of Rhodes semisimple semigroups where every minimal degree of a faithful representation is afforded by a semisimple action.

\section{On the left minimal degree of a semigroup}

 In this section we look at representations of a semigroup by partial and total functions acting on the left.   It is often convenient to regard left actions of a semigroup $S$ as right actions by the opposite semigroup $S^{op}$. We have defined
$\lambda(S) = \mu(S^{op})$ and $l(S)=m(S^{op})$ as the minimal degrees of a left action of $S$ by total and partial functions respectively. Since we know that an inverse semigroup (so in particular for a group) $S$ is isomorphic to $S^{op}$ by the inversion map sending an element $s \in S$ to $s^{-1}$, it follows that
$\lambda(S) = \mu(S)$ and $l(S)=m(S)$. We show that the situation for semigroups is very different. We first provide a well-known  upper bound.

The monoid $PT_n^{op}$ of all partial transformations of $[n]$ acting on the left has an  action  by partial functions on the right of $2^{[n]} -\{\emptyset\}$ of non-empty subsets of $[n]$ by inverse image. More precisely, if $f \in PT_{n}^{op}$ and $X \subseteq [n]$ is a non-empty set, define $Xf = f^{-1}(X)$ if this is non-empty and undefined otherwise. It is clear that this is a right action by partial functions and we claim that it is faithful. For if $f,g \in PT_n^{op}$ and $Xf = Xg$ for all non-empty subsets $X$, then in particular, $\{i\}f=f^{-1}(i)=g^{-1}(i)=\{i\}g$ for all $i \in [n]$. That is $f$ and $g$ have the same domain and the same fibers
over elements in their ranges and thus $f=g$.

Moreover the monoid $PT_n^{op}$ of all partial transformations of $[n]$  satisfies the hypotheses of Corollary~\ref{c:aggm} with $J$ the $\J$-class of rank $1$ partial maps.  In this case, there are $2^n-1$ $\mathscr L$-classes corresponding to the possible domains of a rank $1$ partial mapping and $n$ $\mathscr{R}$-classes corresponding to their ranges. So the minimal faithful degree of $PT_n^{op}$ by partial mappings is $2^n-1$ and the minimal degree by total mappings is $2^n$. We see that this is the worst behavior possible and give a generalization.

\begin{Thm}
Let $S$ be a semigroup. Then $l(S) \leq 2^{m(S)}-1$ and $\lambda(S) \leq 2^{\mu(S)}$. Furthermore, for every $k$ with $n \leq k \leq 2^{n}-1$, there is a semigroup $S$
with $m(S) = n$ and $l(S) = k$ and for every $k$ with $n \leq k \leq 2^{n}$ there is a semigroup with $\lambda(S) =n$ and $\mu(S) = k$.
\end{Thm}
\begin{proof}
The first part of the Theorem follows from our discussion about $PT_{n}^{op}$ above. We will prove the second part of the Theorem for the case of $m(S)$ and $l(S)$. The proof for $\lambda(S)$ and $\mu(S)$ follows easily from the proof.

Pick a set $\{X_{1},\ldots, X_{k-n}\}$ of $k-n$ distinct non-empty subsets of $[n]$ all of cardinality at least 2. Let $X$ be the $n \times (k-n)$ matrix over $\{0,1\}$
whose $i^{th}$ column is the characteristic vector of $X_{i}$, for $i=1,2, \ldots, k-n$. Let $C$ be the $n \times k$ matrix $=[I_{n}|X]$, where $I_{n}$ is the identity matrix. Then the presence of the identity matrix ensures that all rows of $C$ are non-zero and distinct rows are not equal. The columns of $C$ are all distinct by construction. It follows from Theorem~\ref{t:proportional} that $\mathscr{M}^{0}(1,[k],[n],C)$ is an $\mathsf{AGGM}$ semigroup and thus Corollary~\ref{c:aggm}
implies that $m(S)=n$ and $l(S)=k$ as claimed.
\end{proof}

Let $T_n$ be the full transformation monoid acting on the right of $[n]=\{1,\ldots, n\}$ with $n\geq 2$. Note that $T_{n}$ is an $\RM$ semigroup, since it acts faithfully on the right of its minimal ideal of constant maps. Similarly, $T_{n}^{op}$ is an $\LM$ semigroup. Neither of these are $\GGM$ semigroups by Theorem~\ref{t:proportional} and thus are not Rhodes semisimple. One can compute that the congruence $\equiv_{\GGM}$ is the Rees quotient identifying all constant maps to a single point in both cases. Thus the previous results on the left minimal degree do not apply. Nonetheless, we prove that $\lambda(T_{n})$ = $2^{n}$, as for the monoid of partial
functions. That is, one can not embed $T_{n}^{op}$ into $T_{m}$ if $m < 2^{n}$.

We use the inverse action as in the case of $PT_{n}^{op}$, but retaining  the empty set, which is a sink state, this time.   So we get a $T_{n}^{op}$-set with underlying set $2^{[n]}$. As above this is a faithful $S$-set. It will be convenient to view the action in an equivalent way.  Let $e\in T_n^{op}$ be an idempotent with image $\{1,2\}$ and let $c$ be the constant map to $1$ and $d$ the constant map to $2$.  Note that $eT_n^{op}$ consists of those mappings with image contained in $\{1,2\}$.  There is a bijection between $2^{[n]}$ and $eT_n^{op}$ taking a subset $A$ to the unique mapping $f_A$ sending $A$ to $1$ and $[n]\setminus A$ to $2$.  Moreover, if $g\in T_n^{op}$, then $f_Ag = f_{g\inv (A)}$.  Thus the right action of $T_n^{op}$ on the right ideal $eT_n^{op}$ is isomorphic to the action on $2^{[n]}$ by inverse images.

\begin{Thm}\label{t:tnop}
Suppose $n\geq 2$.
Let $X$ be a faithful $T_n^{op}$-set.  Then it contains an invariant subset isomorphic to a  copy of $T_n^{op}$ acting on $2^{[n]}$.  Thus the minimal degree of a faithful action of $T_n^{op}$ by total transformations is $2^n$.
\end{Thm}
\begin{proof}
Suppose now that $X$ is a faithful right $T_n^{op}$-set.  Then we show that there is a $T_n^{op}$-equivariant embedding of $eT_n^{op}$ into $X$.  It will follow that $eT_n^{op}$ is the unique minimal faithful right $T_n^{op}$-set up to isomorphism.  Indeed, let $x\in X$ such that $xc\neq xd$.   Then there is a $T_n^{op}$-equivariant mapping $\sigma\colon T_n^{op}\to X$ given by $\sigma(g) = xg$.  We claim that $\sigma|_{eT_n^{op}}$ is injective.  For suppose $f,g\in eT_n^{op}$ with $f\neq g$.  Then, without loss of generality, we can find $i$ with $f(i)=1$ and $g(i)=2$.  Then if $h$ is the constant map with image $i$, we have $fh=c$ and $gh=d$.  Thus $xfh=xc\neq xd=xgh$, and so we have $xf\neq xg$.  Thus $\sigma$ is injective.
\end{proof}

The astute reader will note that the proof of Theorem~\ref{t:tnop} applies equally well to any subsemigroup of $T_n^{op}$ containing $eT_n^{op}$ for some rank $2$ idempotent $e$.

Theorem~\ref{t:tnop} implies the well-known result~\cite{SalomaaReverse} that if $\mathscr A$ is a minimal deterministic $n$-state automaton whose transition monoid is $T_n$, then the minimum size of an automaton accepting the reversal of the language of $\mathscr A$ is $2^{[n]}$. First note that $T_n$ is the syntactic monoid of $L(\mathscr A)$ and hence $T_n^{op}$ is the syntactic monoid of $L(\mathscr A)^{op}$.  Thus the minimal automaton of $L(\mathscr A)^{op}$ has at least $2^n$ states by Theorem~\ref{t:sn.case}.  But the standard construction of an automaton recognizing $L(\mathscr A)^{op}$ is to reverse the edges of $\mathscr A$ and perform the subset construction.  Then the corresponding transformation monoid is then exactly $T_n^{op}$ acting on $2^{[n]}$.

Note that the minimal degree of a faithful action of $T_n^{op}$ by partial transformations acting on the right is $2^{[n]}-1$.  Indeed, the empty set is a sink for the right action and so if we remove it we get a faithful action by partial mappings of degree $2^n-1$.  On the other hand, if we have a faithful action by partial mappings, then adding a sink creates a faithful action by total mappings which must contain at least $2^n$ points by the above.  Therefore, the action must have degree at least $2^n-1$.

\appendix

\section{Semilocal theory and semigroups with faithful completely reducible representations} \label{s:appendix}

The local theory of finite semigroups studies the structure of $\mathscr{J}$-classes via the Rees Theorem and the various partial orders arising from Green's relations. The global theory of finite semigroups studies decompositions of semigroups via various products (direct, semidirect, wreath, block) and division. The archetypical result is the Krohn-Rhodes Theorem. Semilocal theory is concerned with how elements act on the right and left of $\mathscr{J}$-classes. These are the Sch\"utzenberger and related representations associated to $\mathscr{J}$-classes, their interconnections and their relationship to the representation theory of semigroups by (partial) functions and by matrices over a field. The purpose of this appendix is to summarize those parts of the semilocal theory used in this paper for the convenience of the reader. We don't include any proofs, which are readily available in Chapters~7 and~8 of~\cite{Arbib} and in a more modern version in Section~4.6 of~\cite{qtheor}.

A semigroup $S$ is called \emph{right (left) mapping}~\cite{KRannals,qtheor,Arbib} if it contains a ($0$-)minimal ideal $I$ such that
$S$ acts faithfully on the right (left) of $I$. One calls $S$ a \emph{generalized group mapping semigroup} if it is both left mapping and right mapping, and $S$ is called a \emph{group mapping semigroup} if it is generalized group mapping semigroup and $I$ contains a non-trivial group. We often  write that $S$ is an $\RM$ ($\mathsf{LM},\GGM,\mathsf{GM}$) semigroup if it is a right (left, generalized group, group) mapping semigroup.

Let $R$ ($L$) be an $\mathscr{R}$-class ($\mathscr{L}$-class) of $S$. The right Sch\"utzenberger representation on $R$ is the partial $S$-set $R$ ($L$) with action given by $r\cdot s = rs$ if $rs \in R$ and undefined otherwise. The left Sch\"utzenberger representation on $L$ is the
partial left set $S$-set $L$ with action given by $s\cdot l  = sl$ if $sl \in L$ and undefined otherwise.

We gather some basic results about these concepts.
\begin{Thm}
Let $S$ be a finite semigroup.
\begin{enumerate}
    \item If $S$ is $\RM$ or $\mathsf{LM}$, then it has a unique (0-)minimal ideal $I$ that is necessarily regular.
	
	\item Let $R,R'$ be $\mathscr{R}$-classes contained in the same $\mathscr{J}$-class $J$ of $S$. Then the right Sch\"utzenberger
	representation on $R$ is isomorphic to that of $R'$. The dual result holds for a pair of $\mathscr{L}$-classes contained in $J$.
	
	\item $S$ ($S^{op}$) has a faithful transitive action by partial mappings if and only if $S$ is $\RM$ ($\mathsf{LM}$). In this case, the right (left) Sch\"utzenberger representation is a faithful transitive partial $S$-set ($S^{op}$-set) for any $\mathscr{R}$-class $(\mathscr{L}$-class) contained in $I$.
\end{enumerate}
\end{Thm}

Let $J$ be a regular $\mathscr{J}$-class of $S$. We connect the Sch\"utzenberger representations to the two congruences $\equiv_J$ and $\equiv_{\RM,J}$ introduced in Section~\ref{s:rss}.  The semigroup $S$ acts on the right (left) of $J$ by partial functions. Let $\RM_{J}(S)$ ($\mathsf{LM}_{J}(S)$) be the image of $S$ under the corresponding homomorphism to $PT_J$ ($PT_J^{op}$). We have the canonical surjective homomorphisms
$\rho_{J}\colon S \rightarrow \RM_{J}(S)$ ($\lambda_{J}\colon S \rightarrow \mathsf{LM}_{J}(S)$).

We state the following Theorem for $\RM_{J}$. The dual results hold for $\mathsf{LM}_{J}$.

\begin{Thm} Let $S$ be a semigroup and $J$ a regular $\mathscr{J}$-class of $S$.
\begin{enumerate}
\item For all $s,t \in S, \rho_J(s) = \rho_J(t)$ if and only if $s\equiv_{\RM,J}t$. Therefore the quotient semigroup $S/{\equiv_{\RM,J}}$ is isomorphic to $\RM_{J}(S)$

\item Let $R$ be an $\mathscr{R}$-class of $J$. Then for all $s,t \in S$ $s\equiv_{\RM,J}t$ if and only if the restriction of the action of $s$ and $t$ to $R$ define the same partial function on $R$. It follows that $RM_{J}(S)$ is isomorphic to the image of $S$ in the right Sch\"utzenberger representation of $S$ on $R$ for any $\mathscr{R}$-class $R$ of $J$.
\end{enumerate}
\end{Thm}

Now we consider the congruence $\equiv_J$ for a regular $\mathscr{J}$-class of $S$. It is known that the image of $J$ under the morphism
$\rho_{J}$ ($\lambda_{J}$) is a $\mathscr{J}$-class $\rho_{J}(J)$ ($\lambda_{J}(J)$) of $\RM_{J}(S)$ ($\mathsf{LM}_{J}(S)$). It is known that
$\lambda_{\rho_J(J)}\circ \rho_J= \rho_{\lambda_{J}(J)}\circ \lambda_J$. We denote this common homomorphism by $\gamma_{J}$ and the common semigroup  $\mathsf{LM}_{\rho_{J}(J)}(\RM_{J}(S))= \RM_{\lambda_{J}(J)}(\mathsf{LM}_{J}(S))$ by $\GGM_{J}(S)$. We have the following theorem.

\begin{Thm}
The semigroup $\GGM_{J}(S)$ is a $\GGM$ semigroup. Consider $\gamma_{J}\colon S \rightarrow \GGM_{J}(S)$. Then for all $s,t \in S$ we have
$\gamma_{J}(s)=\gamma_{J}(t)$ if and only if $s\equiv_{J}t$. Therefore the quotient semigroup $S/{\equiv_{J}}$ is isomorphic to $\GGM_{J}(S)$.
\end{Thm}

The ($0$-)minimal ideal $I$ of a $\RM$ ($\mathsf{LM}$,$\GGM$) semigroup is a ($0$-)simple semigroup, so is isomorphic to a Rees matrix semigroup
$I \cong \mathscr M^{0}(G,A,B,C)$ (with the appropriate convention if $I$ has no $0$). Here $G$ is a group, $A$ is a set indexing the $\mathscr{R}$-classes of $I$, $B$ is a set indexing the
$\mathscr{L}$-classes of $I$ and $C\colon B \times A \rightarrow G \cup \{0\}$ is the sandwich matrix. We write $C$ as a $B \times A$ matrix. As a set, $I = (A \times G \times B) \cup \{0\}$ with multiplication given by $(a,g,b)(a',g',b')=(a,gC_{ba'}g',b')$ if $C_{ba'} \in G$ and 0 otherwise. Every row and column of $C$ has a non-zero entry.

Two columns of $C_{\ast a},C_{\ast a'}$ of $C$ are right proportional if there is $g\in G$ such that  $C_{ba}g=C_{ba'}$ for all $b \in B$. There is the dual notion of a pair of left proportional rows of $C$. The next theorem can be proved by a direct calculation.

\begin{Thm}\label{t:proportional}
 Let $S$ be a semigroup with a unique ($0$-)minimal ideal isomorphic to $\mathscr{M}^{0}(G,A,B,C)$.
\begin{enumerate}
\item If $S$ is $\RM$, then no two columns of $C$ are right proportional.
\item If $S$ is $\mathsf{LM}$, then no two rows of $C$ are left proportional.
\item If $S$ is $\GGM$ then no two columns of $C$ are right proportional and no two rows of $C$ are left proportional.
\end{enumerate}
Moreover, if $S$ is ($0$-)simple, the converse to each of (1)--(3) holds.
\end{Thm}

The morphism $\RM_J$ identifies two $\mathscr{R}$-classes of $J$ whose columns in the sandwich matrix are right proportional and the dual result for $\mathsf{LM}_J$. The next theorem gives universal properties of the congruences $\equiv_{\RM}$, $\equiv_{\mathsf{LM}}$ and $\equiv_{\GGM}$.

\begin{Thm}
Let $S$ be a semigroup. Then $\equiv_{\RM}$, $\equiv_{\mathsf{LM}}$ and $\equiv_{\GGM}$ are, respectively, the largest congruences on $S$ that are injective on subgroups of $S$ and separate regular $\mathscr{L}$-classes, $\mathscr{R}$-classes and $\mathscr{J}$-classes, respectively.
\end{Thm}

For many more details on these important ideas see the aforementioned references~\cite[Section 4.6]{qtheor},~\cite[Chapter 7-8]{Arbib}. We close the appendix with Rhodes's beautiful theorem calculating the intersection of all congruences of morphisms of a semigroup $S$ to an irreducible semigroup of matrices over the complex numbers. This congruence is called the (complex) Rhodes Radical of $S$. For a generalization to characteristic $p>0$, see~\cite{repbook}.

\begin{Thm}
Let $S$ be a semigroup. The Rhodes Radical is equal to the congruence $\equiv_{\GGM}$. Therefore, $S$ has a faithful completely reducible complex representation if and only if $\equiv_{\GGM}$ is the identity congruence.
\end{Thm}

\section*{Acknowledgments}

The authors thank the referee for a thorough reading of the paper. We also thank Peter Cameron for a discussion of the properties of the outer automorphism of $S_6$.

\bibliographystyle{abbrv}
\bibliography{standard2stu}

\end{document}